\documentclass[11pt]{amsart}

\usepackage{amscd}
\usepackage{amsmath, amssymb}
\usepackage{amsfonts}
\newcommand{\de}{\partial}
\newcommand{\db}{\overline{\partial}}

\newcommand{\ddbar}{\sqrt{-1} \partial \overline{\partial}}
\newcommand{\Ric}{\mathrm{Ric}}
\newcommand{\ov}[1]{\overline{#1}}

\newcommand{\tr}[2]{\mathrm{tr}_{#1}{#2}}

\newcommand{\vp}{\varphi}

\newcommand{\ve}{\varepsilon}

\numberwithin{equation}{section}

\renewcommand{\leq}{\leqslant}
\renewcommand{\geq}{\geqslant}
\renewcommand{\le}{\leqslant}
\renewcommand{\ge}{\geqslant}

\begin{document}
\newcounter{remark}
\newcounter{theor}
\setcounter{remark}{0}
\setcounter{theor}{1}
\newtheorem{claim}{Claim}
\newtheorem{theorem}{Theorem}[section]
\newtheorem{lemma}[theorem]{Lemma}
\newtheorem{conj}[theorem]{Conjecture}
\newtheorem{conjq}[theorem]{Conjecture/Question}
\newtheorem{corollary}[theorem]{Corollary}
\newtheorem{proposition}[theorem]{Proposition}
\newtheorem{question}[theorem]{Question}
\newtheorem{defn}{Definition}[theor]

\newenvironment{example}[1][Example]{\addtocounter{remark}{1} \begin{trivlist}
\item[\hskip
\labelsep {\bfseries #1  \thesection.\theremark}]}{\end{trivlist}}

\title{The Chern-Ricci flow}

\author[V. Tosatti]{Valentino Tosatti}
\address{Department of Mathematics and Statistics, McGill University, Montr\'eal, Qu\'ebec H3A 0B9, Canada}
\email{valentino.tosatti@mcgill.ca}

\author[B. Weinkove]{Ben Weinkove}
\address{Department of Mathematics, Northwestern University, 2033 Sheridan Road, Evanston, IL 60208}
\email{weinkove@math.northwestern.edu}
\begin{abstract}
We give a survey on the Chern-Ricci flow, a parabolic flow of Hermitian metrics on complex manifolds.  We emphasize open problems and new directions.
 \end{abstract}

\maketitle

\section{Introduction}

On a K\"ahler manifold $(M, \omega_0)$, a solution of the Ricci flow starting from $\omega_0$ is given by a smooth family of K\"ahler metrics $\omega=\omega(t)$ satisfying
\begin{equation} \label{RF}
\left\{
                \begin{aligned}
                  &\frac{\de \omega}{\de t}  =-\mathrm{Ric}(\omega), \quad 0\leq t<T,\\
                  &\omega(0)=\omega_0,
                \end{aligned}
              \right.
\end{equation}
for $T\in (0,\infty]$.  Here $\textrm{Ric}(\omega)$ is the $(1,1)$ form associated to the Ricci curvature of the underlying metric $g$ associated to $\omega$ and is given explicitly as
\begin{equation} \label{Ric}
\textrm{Ric}(\omega) \overset{{\rm loc}}{:=} - \ddbar \log \det g.
\end{equation}

The equation (\ref{RF}) is known as the \emph{K\"ahler-Ricci flow}.  It was pointed out by Hamilton that his Ricci flow \cite{H} preserves the K\"ahler condition of a metric (at least in the compact case, see Shi \cite{Shi2} for the complete noncompact case).  H.D. Cao \cite{Cao} first studied the K\"ahler-Ricci flow in its own right, and observed that it is equivalent to a scalar parabolic complex Monge-Amp\`ere equation.

More recently there has been interest in the study of nonlinear PDEs, including complex Monge-Amp\`ere equations and geometric flows, in the setting of \emph{non-K\"ahler} complex manifolds.  All complex manifolds admit Hermitian metrics.  We will identify such a Hermitian metric $g$, as in the K\"ahler case, with its $(1,1)$ form $\omega = \sqrt{-1} g_{i\ov{j}}dz^i \wedge d{\ov{z}}^j$.  In general $\omega$ is not closed.  Moreover, the Ricci curvature of the Levi-Civita connection of $g$ is in general not of type $(1,1)$ and hence Hamilton's Ricci flow will not in general give a flow of Hermitian metrics.  It is natural then to look for another geometric flow of Hermitian metrics which specializes to the Ricci flow when $\omega_0$ is K\"ahler.   Many such flows, including the titular \emph{Chern-Ricci flow}, have been proposed and investigated (see for example \cite{GSt, LY, PPZsurvey, PPZ2, St, STi, STi2} and Section \ref{other} for some discussions of other flows).

 Indeed, the level of interest on flows on complex manifolds has reached the point where it has a newly designated classification code in the 2020 Mathematics Subject Classification, ``53E30 Flows related to complex manifolds (e.g., K\"ahler-Ricci flows, Chern-Ricci flows)''.

The Chern-Ricci flow was first studied by Gill \cite{G} in a special setting, and then introduced and studied in general by the authors \cite{TW}.  It is arguably the simplest to write down.  In fact, the Chern-Ricci flow is formally just (\ref{RF}), (\ref{Ric}) above.  The difference is that now $\omega(t)$ are a family of \emph{Hermitian} metrics on the complex manifold $M$, with $\omega_0$ a given Hermitian metric.  The $\omega(t)$ satisfy as above
\begin{equation} \label{CRF}
\left\{
                \begin{aligned}
                  &\frac{\de \omega}{\de t}  =-\mathrm{Ric}(\omega), \quad 0\leq t<T,\\
                  &\omega(0)=\omega_0,
                \end{aligned}
              \right.
\end{equation}
for $T\in (0,\infty]$.  The only difference is that now $$\textrm{Ric}(\omega) \overset{{\rm loc}}{:=} -\ddbar \log \det g$$ is no longer the Ricci curvature of the Levi-Civita connection of $g$, but rather the \emph{Chern-Ricci form} of the Hermitian metric $g$.  The Chern-Ricci form can be identified with a Ricci curvature tensor of the Chern connection of $g$.

We end the introduction by giving a brief outline of this article.

\medskip

\noindent
\emph{Reduction to a scalar equation.}  In Section \ref{sect_pde} we explain how the Chern-Ricci flow can be reduced to a scalar equation of parabolic complex Monge-Amp\`ere type.  This is a key feature of the Chern-Ricci flow which, as for the K\"ahler-Ricci flow, allows us to make use of analytic tools from the study of complex Monge-Amp\`ere equations to prove long time existence and convergence results.

\bigskip

\noindent
\emph{Examples.} On locally homogeneous complex manifolds one can construct explicit solutions to the Chern-Ricci flow.  We explain this briefly in Section \ref{sectex}.

\medskip

\noindent
\emph{Compact complex surfaces.}   The study of the Chern-Ricci flow on compact complex surfaces is the main focus of this survey.   A rather  detailed picture has emerged of the known and expected behavior of the Chern-Ricci flow on \emph{minimal} (i.e. no smooth rational curves of self-intersection $-1$) compact complex surfaces.  This is explained in Section \ref{sectsurf1}.  There has been great success so far in understanding the case of minimal surfaces with Kodaira dimension greater than or equal to zero, while major challenges remain for those of negative Kodaira dimension.  In Section \ref{sectionnmcs} we explain what is known so far about non-minimal complex surfaces.

\medskip
\noindent
\emph{Singularities.}  In Sections \ref{sectfin} and \ref{sectinfin} we discuss the nature of singularity formation for the Chern-Ricci flow in the case of finite time and infinite time singularities respectively.

\medskip
\noindent
\emph{Further questions and directions.}  Finally in Section \ref{sectfurther} we discuss, rather informally, some further questions and new directions for the study of the Chern-Ricci flow and other related flows.   In particular we highlight a success of Lee-Tam \cite{LeeTam} on using the Chern-Ricci flow to construct solutions of the K\"ahler-Ricci flow on complete \emph{non-compact} manifolds, with an application to Yau's Uniformization Conjecture.\\

\noindent
{\bf Acknowledgments. }Research supported in part by   NSF grants DMS-1903147 (VT) and DMS-2005311 (BW). Part of this work was carried out while the first-named author was visiting the Center for Mathematical Sciences and Applications at Harvard University and the second-named author was visiting the Department of Mathematical Sciences at the University of Memphis, whom we would like to thank for  their kind support and hospitality.

\section{The parabolic complex Monge-Amp\`ere equation}\label{sect_pde}

On a compact complex manifold, the Chern-Ricci flow is equivalent to a parabolic equation for a scalar function \cite{G, TW}, just like the K\"ahler-Ricci flow.  Namely, there exists a solution $\omega=\omega(t)$ of (\ref{CRF}) for $t$ in $[0,T)$ if and only if there exists a solution $\varphi=\varphi(t)$ for $t$ in $[0,T)$ of
\begin{equation} \label{vpe}
\left\{\begin{aligned}
& \frac{\partial \varphi}{\partial t} =  \log \frac{(\omega_0 - t \textrm{Ric}(\omega_0) +\ddbar \varphi)^n}{\omega_0^n}, \\
 & \omega_0 - t \textrm{Ric}(\omega_0) +\ddbar \varphi >  0, \\
 & \varphi(0)=  0.
 \end{aligned} \right.
\end{equation}
Indeed, given $\varphi(t)$ solving (\ref{vpe}) then $\omega(t) := \omega_0 - t\textrm{Ric}(\omega_0) + \ddbar \varphi(t)$ satisfies
$$\frac{\partial}{\partial t} \omega(t)  = - \textrm{Ric}(\omega_0) + \ddbar \log \frac{ \omega(t)^n}{\omega_0^n} = - \textrm{Ric}(\omega(t)),$$
and hence solves (\ref{CRF}).

Conversely, given a solution $\omega(t)$ of (\ref{CRF}) we define  $\varphi(x,t)$, for each fixed $x \in M$, by
$$\vp(x,t)=\int_0^t\log\frac{(\omega(x,s))^n}{(\omega_0(x))^n}ds,$$
which thus satisfies the ODE
$$\frac{\partial}{\partial t} \varphi (x,t) = \log \frac{(\omega(x,t))^n}{(\omega_0(x))^n}, \quad \varphi(x,0)=0.$$
Then one can compute that
$$\frac{\partial}{\partial t} (\omega(t) - \omega_0 + t\textrm{Ric}(\omega_0) -\ddbar \varphi)=0,$$
$$\left(\omega(t) - \omega_0 + t\textrm{Ric}(\omega_0) -\ddbar \varphi\right)|_{t=0}=0,$$
and so $\omega(t) = \omega_0 -  t\textrm{Ric}(\omega_0) +\ddbar \varphi>0$ for all $t\in [0,T)$, giving that $\varphi(t)$ solves (\ref{vpe}).

We observe that (\ref{vpe}) is a scalar parabolic equation (of complex Monge-Amp\`ere type) and hence, since $M$ is compact, the Chern-Ricci flow always admits a unique solution on some maximal time interval $[0,T)$ for some $0 <T \le \infty$, by standard parabolic theory (see e.g. \cite[Theorem 3.5]{Tos2} which can be immediately adapted to this setting).  In \cite{TW} it was shown that in fact $T$ can be  characterized as:
\begin{equation} \label{T}
T =  \sup \left\{  t > 0 \, | \, \exists \psi \in C^{\infty}(M)\ \textrm{with } \omega_0 - t\textrm{Ric}(\omega_0) + \ddbar \psi >0 \right\}.
\end{equation}
Note that to see that $T$ is bounded from above by the right hand side of (\ref{T}) is an immediate consequence of (\ref{vpe}).  The reverse inequality requires \emph{a priori} estimates.

We remark that when $\omega_0$ is K\"ahler, the result (\ref{T}) is due to Tsuji \cite{Ts2} and Tian-Zhang \cite{TZ} (see also \cite{Cao, Ts}), and in this case the condition can be expressed in terms of the positivity of the cohomology class $[\omega_0 - t \textrm{Ric}(\omega_0)] = [\omega_0] + t 2\pi c_1(K_M)$, where $K_M$ is the canonical bundle of $M$.

It turns out that in many cases it is rather straightforward to compute $T$, especially in the case of complex surfaces which we will discuss in the next two sections.
In particular, we mention a useful result, proved in \cite[Theorem 2.1]{TWY}:  we have $T=\infty$ if and only if the canonical bundle $K_M$ is \emph{nef} (in the analytic sense).
This is in fact an easy consequence of (\ref{T}) and the definition of $K_M$ being nef, which can be stated as follows:  for every $\ve>0$ there exists $\psi_{\ve}\in C^{\infty}(M)$ such that $-\textrm{Ric}(\omega_0) + \ddbar \psi_{\ve} >- \ve \omega_0$.

Observe that, as is the case for the K\"ahler-Ricci flow, the condition that $T=\infty$ is independent of the initial metric $\omega_0$.

Lastly, we briefly discuss the setting of Gill's paper \cite{G} (see also \cite{Sm} for a recent generalization). There, he considered the parabolic complex Monge-Amp\`ere equation
\begin{equation} \label{vpe2}
\left\{\begin{aligned}
& \frac{\partial \varphi}{\partial t} =  \log \frac{(\omega_0 +\ddbar \varphi)^n}{e^F\omega_0^n}, \\
 & \omega_0 +\ddbar \varphi >  0, \\
 & \varphi(0)=  0.
 \end{aligned} \right.
\end{equation}
on an arbitrary compact Hermitian manifold $(M,\omega_0)$, where $F\in C^\infty(M)$ is also arbitrary, and proved that this equation always has long-time existence and convergence (after normalization) to the solution of an elliptic complex Monge-Amp\`ere equation, which had been previously solved by the authors \cite{TW2}. This is the Hermitian analogue of the result of Cao \cite{Cao} for who gave a parabolic approach to Yau's theorem on compact K\"ahler manifolds \cite{Y}.  The PDE \eqref{vpe2} (with suitable choice of $F$) is equal to the Chern-Ricci flow equation \eqref{vpe} only in the case when $\Ric(\omega_0)=\ddbar\psi$ for some $\psi\in C^\infty(M)$. These manifolds with vanishing {\em first Bott-Chern class} \cite{Tos} will be discussed below.

The general Chern-Ricci flow \eqref{CRF} was introduced by the authors in \cite{TW}, who also proved the above equivalence with \eqref{vpe} and the characterization \eqref{T} of the maximal existence time.

\section{Examples}\label{sectex}

One can write down many explicit solutions of the Chern-Ricci flow on compact complex manifolds which are locally homogeneous (equivalently, they admit a geometric structure in the sense of Thurston \cite{Wa}). Indeed, let $(M^n,\omega)$ be a locally homogeneous compact Hermitian manifold. Its universal cover is then biholomorphic to $\Gamma\backslash G$ where $G$ is a Lie group with a left-invariant complex structure, and $\Gamma\subset G$ is a discrete subgroup, and $\omega$ is induced by a left-invariant Hermitian metric on $G$.

If $(M,\omega_0)$ is locally homogeneous, then it was observed in \cite{LR} that
\begin{equation}\label{evolv}
\omega(t):=\omega_0-t\mathrm{Ric}(\omega_0),
\end{equation}
is always a solution of the Chern-Ricci flow \eqref{CRF}, which is still locally homogeneous. Now, locally homogeneous compact complex surfaces (that is, of complex dimension $2$) were classified by Wall \cite{Wa}, they are all minimal, and the non-K\"ahler ones are exactly diagonal Hopf surfaces with eigenvalues of the contraction with equal norm, and all Inoue surfaces, Kodaira surfaces and non-K\"ahler minimal properly elliptic surfaces. There are $6$ possibilities for the Lie group $G$.

The evolved metrics \eqref{evolv} were explicitly computed in all these cases in \cite{TW3} (except Kodaira surfaces, for which the behavior of the flow was already determined in \cite{G}, but these can be easily treated in the same way), where the Gromov-Hausdorff limit of $(M,\omega(t))$ as $t\to T$ was also determined in each case. Some of these will be discussed below.

\section{Minimal complex surfaces} \label{sectsurf1}

Let $M$ be a compact complex surface (namely, of complex dimension 2).  For this and the next section we will discuss what is known, and not known, about the behavior of the Chern-Ricci flow on $M$.  It is convenient for us to make an assumption that the initial Hermitian metric $\omega_0$ is \emph{Gauduchon}, which means that
\begin{equation} \label{ddbc}
\partial \overline{\partial} \omega_0 =0,
\end{equation}
or in other words $\omega_0$ is $\partial \ov{\partial}$-closed, or \emph{pluriclosed}.  Note that for general dimension $n\ge 2$ the Gauduchon condition is $\partial \ov{\partial}\omega_0^{n-1}=0$, a \emph{non-linear} equation in $\omega_0$.

It is natural to restrict to Gauduchon metrics in complex dimension 2.  Indeed, the well-known theorem of Gauduchon \cite{Ga} states that there exists a unique (up to scaling) Gauduchon metric in the conformal class of any Hermitian metric.  Hence it imposes no assumption on $M$.  On the other hand the $\partial\ov{\partial}$-closed condition  (\ref{ddbc}) is preserved along the Chern-Ricci flow, since the Chern-Ricci form is $d$-closed and therefore $\partial\ov{\partial}$-closed.  So it is also natural in the context of the Chern-Ricci flow.  Hence for the rest of this section and the next we will assume (\ref{ddbc}), even though in fact some of the results discussed do not actually require it.

In particular, under this Gauduchon assumption, the maximal existence time $T$ of the Chern-Ricci flow is equal to the first time when either the total volume of $M$ or the volume of some holomorphic curve $C$ with negative self-intersection goes to zero \cite[Corollary 1.4]{TW}, and these volumes can be effectively computed as
$$\mathrm{Vol}(M,\omega(t))=\frac{1}{2}\int_M(\omega_0-t\Ric(\omega_0))^2,\quad \mathrm{Vol}(C,\omega(t))=\int_C(\omega_0-t\mathrm{Ric}(\omega_0)).$$
This allows us to easily compute $T$ on all compact complex surfaces.

In this section we will analyze the case when $M$ is in addition \emph{minimal}, in the sense of complex geometry.  This means that $M$ admits no $(-1)$-curves, namely no smooth rational (i.e. biholomorphic to $\mathbb{P}^1)$ curve $E$ with self-intersection number $E\cdot E=-1$.

It is a classical fact that from every complex surface one can obtain a minimal complex surface after a finite number of birational transformations known as \emph{blow-downs}.  Indeed if $M$ admits a $(-1)$ curve $E$ then one can explicitly construct (see e.g. \cite{bhpv}) a holomorphic surjective map  $\pi:M \rightarrow N$, for $N$ a smooth compact complex surface, with  $\pi(E)=p \in N$ and such that $\pi$ restricted to $M \setminus E$ is a biholomorphism onto $N \setminus  \{p \}$.  If $N$ admits a $(-1)$ curve then repeat.  This process must terminate since the second Betti number of $N$ is strictly less than that of $M$.

To understand complex surfaces then, it makes sense to start with the minimal ones.  In  Section \ref{sectionnmcs} we will discuss the case when $M$ is not minimal.

We divide the minimal complex surfaces into groups according to their \emph{Kodaira dimension}.  This is a birational invariant defined in terms of $H^0(M, K_M^{\ell})$ for $\ell \ge 1$, the complex vector space of global holomorphic sections of the $\ell$th power of the canonical bundle $K_M$ of $M$.
 If all of the spaces $H^0(M, K_M^{\ell})$ for $\ell\ge 1$ are trivial then we define $\textrm{Kod}(M) =-\infty$, and otherwise we let
$$\mathrm{Kod}(M)=\limsup_{\ell\to+\infty}\frac{\log \dim H^0(M, K_M^{\ell})}{\log\ell},$$
which can be shown to be an integer that satisfies $0\leq \mathrm{Kod}(M)\leq 2$.  Geometrically, the Kodaira dimension of $M$ is measuring the dimension of the image of the $\ell$-\emph{canonical (meromorphic) map} $\Phi_{\ell} : M \dashrightarrow \mathbb{P}(H^0(M, K_M^{\ell})^*)$, for $\ell$ sufficiently large and divisible.

There are then four possibilities:  $\textrm{Kod}(M)= 2,1,0, -\infty$.  We will start with $\textrm{Kod}(M)=2$ and work downwards.

In general, it is useful to know that if $M$ is minimal and $\mathrm{Kod}(M)\geq 0$ then $K_M$ is nef  \cite[Corollary III.2.4]{bhpv}.

\subsection{$\textrm{Kod}(M)=2$.}  \label{sectionkod2} In this case $M$ is projective algebraic, a \emph{minimal surface of general type}, and in particular admits a K\"ahler metric.    We consider here the case when $\omega_0$ is \emph{not} K\"ahler.  Since $K_M$ is nef the solution to the Chern-Ricci flow exists for all time starting at any Hermitian $\omega_0$.
 The authors showed in \cite{TW} that when $K_M$ is ample we have smooth convergence
 $$\frac{\omega(t)}{t} \rightarrow \omega_{\textrm{KE}}, \quad \textrm{as } t\rightarrow \infty,$$
 where  $\omega_{\textrm{KE}}$ is the unique K\"ahler-Einstein metric on $M$ satisfying $$\textrm{Ric}(\omega_{\textrm{KE}})=-\omega_{\textrm{KE}}.$$
 Gill \cite{G3} considered the case when $K_M$ is not ample.  He showed that $\omega(t)/t \rightarrow \omega_{\textrm{KE}}$ in $C^{\infty}_{\textrm{loc}}(M \setminus S)$ as $t \rightarrow \infty$ where $S$ is the (nonempty) union of all $(-2)$ curves on $M$ and $\omega_{\textrm{KE}}$ is a K\"ahler-Einstein metric with orbifold singularities on $M$.   We note that the results of \cite{TW, G3} do not require $\omega_0$ to be Gauduchon, and also hold in higher dimension, except that in higher dimensions $\omega_{\rm KE}$ has more complicated singularities along an analytic subvariety $S$, although it has continuous local potentials by \cite{Ngu}. The subvariety $S$ is by definition the union of all positive-dimensional irreducible analytic subvarieties $V\subset M$ such that
 $(V\cdot K_M^{\dim V})=0,$ and this union is itself a closed analytic subvariety by the main result of \cite{CT}.

When $\omega_0$ is K\"ahler these results were established by Cao \cite{Cao} in the case of $K_M$ ample and Tsuji \cite{Ts} (see also Tian-Zhang \cite{TZ}) when not.

The following problem is thus very natural:
\begin{conj}\label{mingen} Let $M^n$ be a compact complex manifold with $K_M$ nef and $\mathrm{Kod}(M)=n$, and let $\omega(t)$ be the solution of the Chern-Ricci flow \eqref{CRF} starting at an arbitrary Hermitian metric $\omega_0$. Then we have
$$\mathrm{diam}\left(M,\frac{\omega(t)}{t}\right)\leq C,$$
for all $t$ sufficiently large, and
$$\left(M,\frac{\omega(t)}{t}\right)\rightarrow (Z,d), \quad \textrm{as } t \rightarrow \infty,$$
in the sense of Gromov-Hausdorff (namely, convergence  as metric spaces), where $(Z,d)$ denotes the metric completion of $(M \setminus S,\omega_{\rm KE})$.
\end{conj}
These results were shown in the case when $\omega_0$ is K\"ahler and $n=2$ by Guo-Song-Weinkove \cite{GSW} and Tian-Zhang \cite{TZ2} independently. The Tian-Zhang result also applies when $n=3$, and the case of general $n$ was later settled by Wang \cite{Wang}.

Conjecture \ref{mingen} seems more approachable when $n=2$, where one could try to adapt the method of \cite{GSW}, which unlike \cite{TZ2,Wang} does not rely much on the coincidence of the Chern-Ricci form with the Riemannian Ricci curvature for K\"ahler metrics.

\subsection{$\textrm{Kod}(M)=1$.}  All these surfaces admit a surjective holomorphic map $f: M \rightarrow \Sigma$ onto a compact Riemann surface $\Sigma$, such that the generic fiber of $f$ is a torus (i.e. an elliptic curve, after choosing an origin).  These surfaces are called \emph{minimal properly elliptic surfaces}.  We are interested mainly in the case when $M$ does not admit any K\"ahler metric.  It turns out that these non-K\"ahler minimal properly elliptic surfaces are structurally quite simple:  they are either elliptic bundles over a Riemann surface $\Sigma$ (locally $f^{-1}(U)$ is biholomorphic to $U \times F$, for $U$ a small neighborhood in $\Sigma$ and $F$ a torus), or have a finite cover which is one.  Tosatti-Weinkove-Yang \cite{TWY} considered the Chern-Ricci flow on such manifolds.  As explained earlier, the canonical bundle $K_M$ is nef and so the Chern-Ricci flow exists for all time.

For simplicity of the discussion, assume $f:M \rightarrow \Sigma$ is a non-K\"ahler elliptic bundle over $\Sigma$, which is necessarily a Riemann surface of genus at least 2.   It was shown in \cite{TWY} that
 $$\frac{\omega(t)}{t} \rightarrow f^* \omega_{\Sigma}, \quad \textrm{as } t\rightarrow \infty,$$
 in the $C^0(M, g_0)$ topology where $\omega_{\Sigma}$ is the K\"ahler-Einstein metric on the Riemann surface $\Sigma$, satisfying
 $$\textrm{Ric}(\omega_{\Sigma})=-\omega_{\Sigma}.$$
This result  implies in particular that the diameter of each torus fiber tends to zero uniformly and $(M, \omega(t)/t)$ converges to $(S, \omega_{\Sigma})$ in the Gromov-Hausdorff sense.  One can also say more:  $\omega(t)$ restricted to each fiber $f^{-1}(y)$ converges in the $C^1$ topology to the \emph{semi-flat} metric $\omega_{\textrm{flat},y}$ on $f^{-1}(y)$.  The semi-flat metric $\omega_{\textrm{flat},y}$ is defined to be the unique flat metric on the fiber $f^{-1}(y)$ lying in the cohomology class $[\omega_0|_{f^{-1}(y)}]$ (note that $\omega_0$ is not a closed form, but its restriction to the $1$-dimensional fiber is). Very recently, Angella-Tosatti \cite{AT} improved this last convergence result, so that $\omega(t)|_{f^{-1}(y)}\to \omega_{\textrm{flat},y}$ smoothly on $f^{-1}(y)$.

For a general non-K\"ahler properly minimal elliptic surface, one can apply the above result to a finite cover and obtain, for example, Gromov-Hausdorff convergence of $(M, \omega(t)/t)$ to $(\Sigma, d_\Sigma)$ where now $d_\Sigma$ is now the distance function induced by the orbifold K\"ahler-Einstein metric $\omega_{\Sigma}$ on $\Sigma$.

Later, Kawamura \cite{Kaw} improved the convergence of $\omega(t)/t$ to $f^*\omega_\Sigma$ to $C^\alpha(M, g_0), 0<\alpha<1,$ in the case when  $\omega_0$ is of the form $\omega_{\rm TV}+\ddbar\psi$ for some $\psi\in C^\infty(M)$, where $\omega_{\rm TV}$ is an explicit Gauduchon metric on $M$ constructed by Tricerri \cite{Tr} and Vaisman \cite{Va}. Again very recently this was improved to smooth convergence by Angella-Tosatti \cite{AT}, who also proved that in this case the Chern Riemann curvature tensor of $\omega(t)/t$ remains uniformly bounded.

These results strongly suggest that the following should hold:
\begin{conj}\label{prop_nk} Let $f:M\to \Sigma$ be an elliptic bundle over a Riemann surface $\Sigma$ of genus at least $2$, and let $\omega(t),t\geq 0$ be the solution of the Chern-Ricci flow \eqref{CRF} on $M$ starting at an arbitrary Gauduchon metric $\omega_0$. Then
$$\frac{\omega(t)}{t}\to f^*\omega_\Sigma, \quad \textrm{as } t \rightarrow \infty,$$
smoothly on $M$ and $(M,\omega(t)/t)\to (\Sigma,d_\Sigma)$ in Gromov-Hausdorff. Furthermore, the Chern Riemann curvature tensor of $\frac{\omega(t)}{t}$ satisfies the uniform bound
$$\sup_M \left|\mathrm{Rm}\left(\frac{\omega(t)}{t}\right)\right|_{g(t)/t}\leq C,$$
for all $t$ sufficiently large. The same results hold for $f:M\to \Sigma$ a general non-K\"ahler minimal properly elliptic surface, up to passing to a finite cover which makes it an elliptic bundle.
\end{conj}
As we said, Conjecture \ref{prop_nk} was recently proved by Angella-Tosatti \cite{AT} when $\omega_0$ is of the form $\omega_{\rm TV}+\ddbar\psi$ for some $\psi\in C^\infty(M)$, but the general case remains open. See however Question \ref{higherord} below for a local higher order result that would imply Conjecture \ref{prop_nk}.

As discussed above, the non-K\"ahler minimal surfaces of Kodaira dimension one have a relatively simple structure.  In general, minimal properly elliptic K\"ahler surfaces  still admit an elliptic fibration $f:M\to\Sigma$ which may have singular/multiple fibers and this leads to substantial new difficulties when studying the long-time behavior of the flow, which are particularly acute when $\omega_0$ is not K\"ahler.  Furthermore, if we let $S\subset M$ be the union of all singular/multiple fibers, then the fibers $f^{-1}(y),y\in \Sigma\backslash f(S)$ are all tori, but unlike the previous case they are not mutually biholomorphic in general, so there is variation of the complex structure of the smooth fibers. This variation is encoded in a {\em Weil-Petersson form} $\omega_{\rm WP}\geq 0$ on $\Sigma\backslash f(S)$, and the limiting K\"ahler-Einstein metric above is now replaced by a K\"ahler metric $\omega_\Sigma$ on $\Sigma\backslash f(S)$ which satisfies the \emph{twisted K\"ahler-Einstein equation}
 \begin{equation} \label{tKE}
 \textrm{Ric}(\omega_{\Sigma})=-\omega_{\Sigma}+\omega_{\rm WP},
 \end{equation}
 and which was constructed by Song-Tian \cite{ST}, who first studied the K\"ahler-Ricci flow on elliptic surfaces.

The following is then the expected picture:

\begin{conj}\label{prop_k} Let $f:M\to \Sigma$ be a general K\"ahler minimal properly elliptic surface, and let $\omega(t),t\geq 0$ be the solution of the Chern-Ricci flow \eqref{CRF} on $M$ starting at an arbitrary Gauduchon metric $\omega_0$. Then
\begin{equation}\label{total}
\frac{\omega(t)}{t}\to f^*\omega_\Sigma, \quad \textrm{as } t\rightarrow \infty,
\end{equation}
locally smoothly on $M\backslash S$, for $\omega_{\Sigma}$ solving (\ref{tKE}), and the Chern Riemann curvature tensor of $\frac{\omega(t)}{t}$ remains locally uniformly bounded on compact subsets of $M\backslash S$.  Furthermore, for each fiber $f^{-1}(y),y\in \Sigma\backslash f(S)$ we have
\begin{equation}\label{fiber}
\omega(t)|_{f^{-1}(y)}\to \omega_{\textrm{flat},y},
\end{equation}
smoothly on $f^{-1}(y)$.
\end{conj}

Analogously with Conjecture \ref{mingen} one also expects:
\begin{conj}\label{prop_k_gh}
In the same setting as Conjecture \ref{prop_k}, we have
$$\mathrm{diam}\left(M,\frac{\omega(t)}{t}\right)\leq C,$$
for all $t$ sufficiently large, and
$$\left(M,\frac{\omega(t)}{t}\right)\rightarrow (Z,d), \quad \textrm{as } t \rightarrow \infty,$$
in the sense of Gromov-Hausdorff, where $(Z,d)$ denotes the metric completion of $(\Sigma\setminus f(S),\omega_{\Sigma})$.
\end{conj}

When $\omega_0$ is K\"ahler, Conjectures \ref{prop_nk} and \ref{prop_k} were settled by \cite{FZ, G2, HT, ToZ}, see also \cite[\S 5.14]{ToZ}, and Conjecture \ref{prop_k_gh} was proved in \cite{STZ}. Thus, it remains to understand what happens when the initial metric $\omega_0$ is Gauduchon and non-K\"ahler. One of the main obstacles to following the arguments in the K\"ahler case is that so far the ``parabolic Schwarz Lemma'' from \cite{ST,Ya} has not been generalized to the non-K\"ahler case (and getting around this issue was also a major concern in \cite{TWY}). Observe also that once \eqref{total} and \eqref{fiber} have been established in $C^0_{\rm loc},$ then smooth convergence in \eqref{fiber} will follow using the method in \cite[Remark 6.2]{AT}.

\subsection{$\textrm{Kod}(M)=0$.} \label{sectionKod0} These surfaces were classified by Kodaira (see \cite{bhpv}), and are either K\"ahler Calabi-Yau surfaces (i.e. $c_1(M)=0$ in $H^2(M,\mathbb{R})$, which are tori, $K3$, Enriques and bielliptic surfaces) or non-K\"ahler (which are primary and secondary Kodaira surfaces). It follows from this that
all such $M$ have $K_M$ holomorphically torsion, which implies that $\Ric(\omega_0)$ is globally $\de\db$-exact. This is a condition known as vanishing first Bott-Chern class (see e.g. \cite{Tos}), which when $M$ is K\"ahler is equivalent to $M$ being Calabi-Yau.

As mentioned in Section \ref{sect_pde}, the behavior of the Chern-Ricci flow on all these manifolds (and in arbitrary dimension) was  dealt with in Gill's first paper on the topic \cite{G}.  Starting at any Hermitian metric $\omega_0$ (the Gauduchon condition is not needed here), the solution of the Chern-Ricci flow exists for all time and converges smoothly to a Chern-Ricci flat metric.  As mentioned earlier, this is the analogue of the result of Cao \cite{Cao} for the K\"ahler-Ricci flow which gave a parabolic approach to Yau's theorem on the existence of K\"ahler Ricci-flat (or \emph{Calabi-Yau}) metrics \cite{Y}.

The rate of smooth convergence is exponentially fast, at least when $\dim M=2$ and $\omega_0$ is Gauduchon. As far as we know, this is not written down explicitly in the literature, but one can adapt the proof in the K\"ahler case which is sketched in \cite[Proof of Theorem 1.5]{ToZ}, using the generalization of the Mabuchi energy from \cite[\S 9]{TW3}.

\subsection{$\textrm{Kod}(M)=-\infty$.} This is the most difficult case, and remains largely open.  By Kodaira's classification \cite{bhpv}, minimal surfaces with negative Kodaira dimension are either $\mathbb{P}^2$, ruled, or of Class VII$_0$ (which means minimal surfaces with $\mathrm{Kod}(M)=-\infty$ and $b_1(M)=1)$. We divide the discussion into four parts.

\bigskip
\noindent
\emph{1. Class VII$_0$ with $b_2=0$: Hopf surfaces.} \medskip \label{sectionhopf}

 \ As we just said, minimal non-K\"ahler surfaces with $\textrm{Kod}(M)=-\infty$ are Class VII$_0$ surfaces, a name that comes from Kodaira's original classification of surfaces (the subscript $0$ indicates minimality).   If $b_2(M)=0$ then these Class VII$_0$ surfaces are classified as either \emph{Hopf surfaces} or \emph{Inoue surfaces} by \cite{Bo,In,LYZ,T0}.

We begin by discussing Hopf surfaces which are easier to describe explicitly but in fact pose a harder challenge in the context of the Chern-Ricci flow. By definition, a compact complex surface is called Hopf if its universal cover is biholomorphic to $\mathbb{C}^2\backslash\{0\}$, and it is automatically minimal. For $(\alpha, \beta) \in \mathbb{C}^2 \setminus \{ 0\}$ with $|\alpha|=|\beta|$ consider the Hopf surface
$$M_{\alpha, \beta} = (\mathbb{C}^2 \setminus \{0 \})/\sim$$ where
$$(z_1, z_2) \sim (\alpha z_1, \beta z_2).$$
The manifold $M_{\alpha, \beta}$ is diffeomorphic to $S^1 \times S^3$ via the map
\begin{equation} \label{map}
\begin{split}
(z_1, z_2) \mapsto & \left( |z_1|^2+|z_2|^2, \frac{(z_1, z_2)}{\sqrt{|z_1|^2+|z_2|^2}} \right)  \\ {} & \in ( \mathbb{R}^+/ (x \sim |\alpha| x) ) \times S^3 \cong S^1\times S^3,
\end{split}
\end{equation}
and hence has vanishing second Betti number (which implies it cannot admit a K\"ahler metric).

We can write down an explicit solution $\omega(t)$ of the Chern-Ricci flow  on $M_{\alpha, \beta}$ starting at the standard Hopf metric
\begin{equation} \label{standard}
\omega_0 = \omega_H:= \frac{\delta_{ij}}{|z_1|^2+|z_2|^2} \sqrt{-1} dz_i \wedge dz_j,
\end{equation}
which via (\ref{map}) is just a standard product metric on $S^1 \times S^3$.  The metric $\omega_H$ is Gauduchon.
We have  \cite{TW}
\[
\begin{split}
\omega(t) = {} & \omega_0 - t \textrm{Ric}(\omega_0) \\
= {} &   \frac{1}{|z_1|^2+|z_2|^2} \left( (1-2t) \delta_{ij} + 2t \frac{\ov{z}_i z_j}{|z_1|^2+|z_2|^2} \right) \sqrt{-1} dz_i \wedge d\ov{z}_j,
\end{split} \]
for $t \in [0,1/2)$ (this also follows from Section \ref{sectex} since $(M,\omega_H)$ is locally homogeneous).
It was shown in \cite{TW3} that as $t\rightarrow 1/2$ we have convergence $(M_{\alpha, \beta}) \rightarrow (S^1, d)$ in the Gromov-Hausdorff sense where $d$ is the distance function on the circle $S^1 \subset \mathbb{R}^2$ with radius  $\frac{\log |\alpha|}{\pi \sqrt{2}}$.  Observe that this is in stark contrast with what happens when $\omega_0$ is K\"ahler, when it is conjectured that all (diameter-normalized) Gromov-Hausdorff limits of the K\"ahler-Ricci flow on compact K\"ahler manifolds are homeomorphic to compact complex analytic varieties (see e.g. \cite{ST4}).

 Interestingly, it is \emph{not} the case that the evolving metrics remain a product of metrics on $S^1 \times S^3$, with the $S^3$ factor scaling to zero as $t$ approaches the singular time, as would be the case for the Ricci flow starting at $g_0$.  Rather, the metrics $\omega(t)$ converge to a non-closed semi-positive $(1,1)$ form whose kernel defines a distribution in the tangent space of $S^3$.  The iterated Lie brackets of this distribution generate the entire tangent space, and this suffices to prove the Gromov-Hausdorff collapse to $S^1$ \cite{TW3}.

A natural conjecture is that the same behavior will occur starting at any initial metric $\omega_0$.  It is known that  as $t$ tends to its maximal existence time $T$, which is finite, the volume $\textrm{Vol}(M_{\alpha, \beta},\omega(t))$ tends to zero \cite{TW}.  If $\omega_0$ is of the form $\omega_H +\ddbar \psi$ for a smooth real function $\psi$, then it was shown in \cite{TW} that there is a uniform bound $|\omega(t)|_{\omega_0} \le C$ on the evolving metric, and this implies sequential convergence at the level of potentials $\varphi(t)$ in $C^{1, \gamma}$ for any $0<\gamma<1$.  However these results are  far from establishing the conjecture.

Edwards \cite{E} considered a more general type of Hopf surface, where the condition $|\alpha|=|\beta|$ is replaced by $0< |\alpha| \le |\beta|$.  In this case the formula (\ref{standard}) no longer defines a metric on $M_{\alpha, \beta}$.  Instead Edwards  considered initial data given by a complicated but explicit metric constructed by Gauduchon-Ornea \cite{GO}, and  established again the estimate $|\omega(t)|_{\omega_0} \le C$ for the evolving metrics.

The above surfaces $M_{\alpha, \beta}$ are known as \emph{primary Hopf surfaces of Class 1}.  There are also  primary Hopf surfaces of Class 0, which are quotients of $\mathbb{C}^2 \setminus \{ 0 \}$ by the relation $(z_1, z_2) \sim (\beta^m z_1+ \lambda z_2^m, \beta z_2)$ for $m$ a positive integer and $\beta, \lambda \in \mathbb{C}$ with $|\beta|>1$ and $\lambda \neq 0$.  There are also \emph{secondary} Hopf surfaces which have a primary Hopf surface as a finite cover.  As far as we are aware, the only result known about these other Hopf surfaces is that the volume tends to zero as $t \rightarrow T<\infty$ \cite{TW}.

\pagebreak[3]
\bigskip
\noindent
\emph{2. Class VII$_0$ with $b_2=0$: Inoue surfaces.} \  \medskip

We now discuss \emph{Inoue surfaces} which can be characterized as those Class VII$_0$ surfaces with $b_2=0$ and no holomorphic curves (Hopf surfaces always contain at least one holomorphic curve).  The Inoue surfaces have been constructed and classified by Inoue \cite{In}.
They are all quotients of $\mathbb{C}\times\mathbb{H}$ (where $\mathbb{H}$ denotes the upper half plane in $\mathbb{C}$) by certain explicit discrete groups, and we refer the reader to \cite{In} or the exposition in \cite{TW3}, and point out just some of the basic facts that we need.

It turns out that the behavior of the Chern-Ricci flow on Inoue surfaces is much better understood than for Hopf surfaces. All Inoue surfaces have nef canonical bundle, and hence the flow exists for all time. As $t \rightarrow \infty$, examples for all Inoue surfaces indicate that, normalized appropriately, the evolving metrics should always collapse the manifold to a circle \cite{TW3}.  We describe now more precisely what is known.

On the universal cover $\mathbb{C}\times\mathbb{H}$ of any Inoue surface $M$, using coordinates $z_1, z_2$ on $\mathbb{C}$ and $\mathbb{H}$ respectively, the Poincar\'e metric
$$\frac{\sqrt{-1}}{(\textrm{Im}z_2)^2} dz_2 \wedge d\ov{z}_2$$ on $\mathbb{H}$ descends to a closed nonnegative real $(1,1)$ form on $M$.   A constant multiple of this  form defines a closed real $(1,1)$ form $\omega_{\infty}$ with
$$0 \le \omega_{\infty} \in c_1(K_M).$$
It was shown by Fang-Tosatti-Weinkove-Zheng \cite{FTWZ} that given any Hermitian metric $\omega_0$ there exists a metric $\tilde{\omega}_0$ in its conformal class so that if $\omega(t)$ is the solution of the Chern-Ricci flow starting at $\tilde{\omega}_0$ then $\omega(t)/t \rightarrow \omega_{\infty}$ uniformly on $M$ and exponentially fast (in fact one can take any initial metric in the $\partial \ov{\partial}$-class of $\tilde{\omega}_0$).  Moreover, $(M, \omega(t)/t) \rightarrow (S^1, d)$ in the sense of Gromov-Hausdorff, where $d$ is the standard distance function on the circle $S^1$ of radius depending on $M$. Very recently, Angella-Tosatti \cite{AT} improved this and obtained the same conclusion for the Chern-Ricci flow starting at any Gauduchon metric on $M$.

Stronger estimates were obtained in \cite{FTWZ} and \cite{AT} under the assumption that the initial metric is in the $\partial\ov{\partial}$-class of a special metric constructed by Tricerri \cite{Tr} and Vaisman \cite{Va} (depending on the type of Inoue surface).  In this case it was shown in \cite{FTWZ} that the metrics $\omega(t)/t$ are uniformly bounded in the $C^1$ topology, giving convergence in $C^{\gamma}$ for any $0<\gamma<1$, and in \cite{AT} this was improved to smooth convergence, together with a uniform bound for the Chern Riemann curvature of $\omega(t)/t$.

To complete the picture for Inoue surfaces, one would expect the following conjectures to hold (cf. \cite[Conjectures 4.1, 4.2]{FTWZ} and \cite{AT}):
\pagebreak[3]
\begin{conj}\label{inoue} Let $M$ be any Inoue surface.
\begin{enumerate}
\item[(1)] The uniform convergence $\omega(t)/t \rightarrow \omega_{\infty}$  holds for the Chern-Ricci flow starting at any Hermitian metric $\omega_0$, not necessarily Gauduchon and without any initial conformal change.
\item[(2)] The convergence $\omega(t)/t \rightarrow \omega_{\infty}$ is in the $C^{\infty}$ topology.
\item[(3)] We have that $|\mathrm{Rm}(\omega(t)/t)|_{g(t)/t}\leq C$ on $M$, for all $t$ sufficiently large.
\end{enumerate}
\end{conj}

It would of course be already interesting to settle parts (2) and (3) under the assumption that $\omega_0$ is Gauduchon.\\

\noindent
\emph{3. Class VII$_0$ with $b_2\ge 1$.} \ \medskip

\ Next we discuss the Class VII$_0$ surfaces with $b_2\ge 1$.   These surfaces remain in general unclassified.  The most well-known open problem in the study of complex surfaces is to  complete this classification.  The \emph{Global Spherical Shell Conjecture} asserts that all such manifolds contain a smooth 3-sphere with a neighborhood biholomorphic to a neighborhood of $S^3 \subset \mathbb{C}^2$, whose complement is connected \cite{Ka}.   Class VII$_0$ surfaces with $b_2 \ge 1$ which admit such a ``global spherical shell'' are known as Kato surfaces and are well-understood (see e.g. \cite{Na}). Teleman has proved the Global Spherical Shell Conjecture when $b_2=1$ \cite{T1} and made further progress when $b_2=2,3$ \cite{T3,T4}.

Little is known about the behavior of the Chern-Ricci flow on Class VII$_0$ surfaces with $b_2\geq 1$, except that the singular time $T$ is finite and that $\textrm{Vol}(M,\omega(t)) \rightarrow 0$ as $t\rightarrow T$ \cite{TW}.  A rather optimistic conjecture \cite[Conjecture 4.6]{FTWZ} is that the limit of the Chern-Ricci flow defines a global spherical shell in the following sense.  As $t \rightarrow T$, the metrics $\omega(t)$ should smoothly converge to a non-closed nonnegative $(1,1)$ form $a$.  The tangent distribution defined by the kernel of $a$ gives, after taking iterated Lie brackets, an integrable 3-dimensional distribution, one leaf of which should be a global spherical shell.   The only evidence in support of this conjecture is that this is exactly what happens for the explicit solutions of the Chern-Ricci flow on the standard Hopf surface as described above, which is indeed a minimal Class VII$_0$ surface with $b_2=0$ which contains a global spherical shell (the image of the unit sphere $S^3\subset\mathbb{C}^2\backslash\{0\}$ under the projection map to the Hopf surface).

   \bigskip
   \noindent
\emph{4. $\mathbb{P}^2$ and ruled surfaces.} \medskip

If a minimal complex surface $M$ with $\textrm{Kod}(M)=-\infty$ admits a K\"ahler metric then it must be either $\mathbb{P}^2$ or a ruled surface.  The behavior of the Chern-Ricci flow on these surfaces, starting with a non-K\"ahler metric, is largely unknown.

In the case of $\mathbb{P}^2$, assuming $\omega_0$ is Gauduchon, the volume $\textrm{Vol}(M,\omega(t))$ shrinks to zero as $t \rightarrow T<\infty$ \cite{TW}.  Interestingly, the volume tends to zero like $(T-t)$, unlike the case of the K\"ahler-Ricci flow which shrinks to a point (a consequence of work of Perelman, see \cite{SeT}) and whose volume tends to zero like $(T-t)^2$.  This suggests that, in contrast, the Chern-Ricci flow might collapse  $\mathbb{P}^2$ to a Riemann surface \cite[Conjecture 4.5]{FTWZ}.

Indeed, since $\mathbb{P}^2$ does not contain any curve with negative self intersection, \cite[Corollary 1.4]{TW} gives that
$T$ is exactly the first time when $\textrm{Vol}(M,\omega(t))$ shrinks to zero, which is the smallest positive root of the quadratic equation
\[
\begin{split}
0= {} & \int_{\mathbb{P}^2}(\omega_0-t\Ric(\omega_0))^2 \\
= {} & \int_{\mathbb{P}^2}(\omega_0-3t\omega_{\rm FS})^2 \\
= {} & =9t^2\int_{\mathbb{P}^2}\omega_{\rm FS}^2-6
t\int_{\mathbb{P}^2}\omega_0\wedge\omega_{\rm FS}+\int_{\mathbb{P}^2}\omega_0^2,
\end{split} \]
using that $\Ric(\omega_0)$ is $\de\db$-cohomologous to $\Ric(\omega_{\rm FS})=3\omega_{\rm FS}$. Here we are writing $\omega_{\rm FS}$ for the standard Fubini-Study metric on $\mathbb{P}^2$. Therefore
$$T=\frac{\int_{\mathbb{P}^2}\omega_0\wedge\omega_{\rm FS}-\sqrt{\left(\int_{\mathbb{P}^2}\omega_0\wedge\omega_{\rm FS}\right)^2-\int_{\mathbb{P}^2}\omega_{\rm FS}^2\int_{\mathbb{P}^2}\omega_0^2}}{3\int_{\mathbb{P}^2}\omega_{\rm FS}^2},$$
and Buchdahl's index-type theorem \cite[Proposition 5]{Bu1} shows that
$$\left(\int_{\mathbb{P}^2}\omega_0\wedge\omega_{\rm FS}\right)^2-\int_{\mathbb{P}^2}\omega_{\rm FS}^2\int_{\mathbb{P}^2}\omega_0^2\geq 0,$$
with equality if and only if there are $\lambda\in\mathbb{R}_{>0}$ and $\vp\in C^\infty(\mathbb{P}^2)$ such that
$\omega_0=\lambda\omega_{\rm FS}+\ddbar \vp$, and this happens if and only if $\omega_0$ is K\"ahler. This shows that $T$ is a double root of the quadratic if and only if $\omega_0$ is K\"ahler, as claimed.

In unpublished work of Xiaokui Yang and the authors, it is shown that for any initial Hermitian metric $\omega_0$ on $\mathbb{P}^n$ we have
$$\omega(t)\leq C\omega_{\rm FS},$$
on $\mathbb{P}^n\times [0,T)$, which in particular proves a uniform diameter upper bound along the flow. This is shown by applying the maximum principle to the quantity $\log\tr{\omega_{\rm FS}}{\omega(t)},$ and using that $\omega_{\rm FS}$ is K\"ahler and has positive bisectional curvature. Apart from this result, which applies in all dimensions, nothing else is known (when $\omega_0$ is non-K\"ahler).

For a ruled surface, it is only known that the volume of the surface shrinks to zero linearly in finite time \cite{TW}.  Surprisingly, the behavior even of the K\"ahler-Ricci flow on a ruled surface, starting with an arbitrary K\"ahler metric, is not completely known (see \cite{Fong, SW0, SSW} for results in special cases and some known estimates in general).

\section{Non-minimal complex surfaces} \label{sectionnmcs}

Let $M$ be a complex surface which admits at least one $(-1)$ curve $E$.  For simplicity of the discussion we will assume first that $\textrm{Kod}(M) \neq -\infty$. Then it is shown in \cite[Theorem 1.5]{TW} that the Chern-Ricci flow starting at any initial Gauduchon metric must have a singularity in finite time $T<\infty$, which is necessarily volume non-collapsing, in the sense that $\mathrm{Vol}(M,\omega(t))\geq c>0$ as $t\to T$. Furthermore, the set of all holomorphic curves in $M$ whose volume is going to  zero at time $T$ is precisely equal to a nonempty finite collection of pairwise disjoint $(-1)$-curves $E_1,\dots, E_k$.

Before discussing the behavior of the Chern-Ricci flow, it will be useful to recall to the reader what is known about the K\"ahler-Ricci flow in this setting.  Informally, the K\"ahler-Ricci flow ``contracts'' these $(-1)$ curves to points and then can be continued in a unique way on the ``blown-down'' manifold \cite{SW1, SW2, SW3}.  This can be repeated until the surface no longer contains any $(-1)$ curves, namely is minimal, in which case we are in the setting of the previous section.  This process is the simplest case of the \emph{analytic minimal model program} proposed by Song-Tian \cite{ST,ST2,ST4}.

We describe now more precisely the results of Song-Weinkove \cite{SW1, SW2}, which made use of estimates from \cite{ST,TZ}.   Let $E_1,\dots,E_k$ be the disjoint $(-1)$ curves as above, then there is a smooth K\"ahler metric $\omega_T$ on $M \setminus \cup E_i$
such that the solution $\omega(t)$ of the K\"ahler-Ricci flow exists on the maximal time interval $[0,T)$ for $0<T<\infty$ and $\omega(t) \rightarrow \omega_T$ in $C^{\infty}_{\textrm{loc}}(M \setminus \cup E_i)$.  Let $\pi: M \rightarrow N$ be the holomorphic map blowing down the $(-1)$ curves $E_1, \ldots, E_k$ to points $p_1, \ldots, p_k$.  Then the metric completion $(N, d)$ of $N \setminus \{p_1, \ldots, p_k\}$ is a compact metric space homeomorphic to $N$ and
$$(M, \omega(t)) \rightarrow (N, d), \qquad \textrm{as } t \rightarrow T,$$
in the  sense of Gromov-Hausdorff, and in particular the diameters of the $(-1)$ curves $E_1, \ldots, E_k$ tend to zero as $t \rightarrow T$.  Moreover, there exists a solution $\omega(t)$ of the K\"ahler-Ricci flow on $N$ starting at time $t=T$ with the property that
$$(N, \omega(t)) \rightarrow (N, d), \qquad \textrm{as } t \rightarrow T^+,$$
in the sense of Gromov-Hausdorff.  Moreover, the solution $\omega(t)$ on $N$ is the unique K\"ahler-Ricci flow ``through the singularity'' at the level of potentials as described in the work of Song-Tian \cite{ST4}.

The natural question is whether the same holds for the Chern-Ricci flow starting at any Gauduchon metric on $M$.  There are some partial results in this direction which we now describe.  It was shown in \cite{TW3} that the first assertion of the paragraph above holds verbatim. Namely, there exist finitely many disjoint $(-1)$ curves $E_1, \ldots, E_k$ on $M$ and a smooth Gauduchon metric $\omega_T$ on $M \setminus \cup E_i$ such that the solution $\omega(t)$ of the Chern-Ricci flow exists on the maximum time interval $[0,T)$ for $0<T<\infty$ and $\omega(t) \rightarrow \omega_T$ in $C^{\infty}_{\textrm{loc}}(M \setminus \cup E_i)$.

To obtain further results, an assumption is made on the initial data $\omega_0$, namely that there exists a smooth 3-form $\beta$ on $N$ such that
$$(*) \qquad \qquad d \omega_0 = \pi^* \beta, \qquad \qquad \qquad $$
where $\pi: M \rightarrow N$ is the map blowing down $E_1, \ldots, E_k$ to points $p_1, \ldots, p_k \in N$.  This assumption is not entirely satisfactory:  it is essentially saying that the torsion of $\omega_0$ along the $(-1)$ curves on $M$ vanishes.  This eliminates a major obstruction in extending the estimates from the K\"ahler-Ricci flow to the Chern-Ricci flow setting.

Under the assumption $(*)$ it was shown in \cite{TW3} that $(M, \omega(t))$ converges in the sense of Gromov-Hausdorff to a compact metric space $(N, d_T)$ which is homeomorphic to $N$, and in particular the diameters of the $(-1)$ curves shrink to zero as $t \rightarrow T$.  It was later proved by Nie \cite{Nie2} and T\^o \cite{To} independently that the Chern-Ricci flow can be continued on $N$ and that one obtains Gromov-Hausdroff convergence backwards in time of $(N, \omega(t))$ to $(N, d_T)$.  The identification of $(N, d_T)$ with the metric completion of $(N \setminus \{ p_1, \ldots, p_k \}, \omega_T)$ remains open.  The more important question is whether one can remove condition $(*)$, or whether in fact it is essential.
For simplicity, we state the following conjecture/question in the setting where there is only one $(-1)$ curve.

\begin{conjq}
Let $M$ be a compact complex surface with $\emph{Kod}(M) \neq -\infty$ and with one $(-1)$ curve $E$ and let $\pi: M \rightarrow N$ be the map blowing down $E$ to $p$ in $N$.  Let $\omega(t)$ be the maximal smooth solution of the Chern-Ricci flow starting at an arbitrary Gauduchon metric $\omega_0$, for $0 \le t< T <\infty$, with $\omega(t) \rightarrow \omega_T$ in $C^{\infty}_{\emph{loc}}(M \setminus E)$ as $t \rightarrow T^-$ where $\omega_T$ is a smooth Gauduchon metric on $M \setminus E$ as above.  Then:
\begin{enumerate}
\item The metric completion of $(M \setminus E, \omega_T)$ is a compact metric space $(N, d_T)$ homeomorphic to $N$ and
$$(M, \omega(t)) \rightarrow (N, d_T), \quad \textrm{as } t\rightarrow T^-,$$
in the sense of Gromov Hausdorff.
\item
There is a unique solution of the Chern-Ricci flow $\omega(t)$ for $t>T$ on $N$ with the property that $\omega(t) \rightarrow (\pi^{-1})^*\omega_T$ in $C^{\infty}_{\emph{loc}}(N \setminus \{ p \})$ as $t \rightarrow T^+$ and
$$(N, \omega(t)) \rightarrow (N, d_T), \quad \textrm{as } t\rightarrow T^+,$$
in the sense of Gromov-Hausdorff.
\end{enumerate}
\end{conjq}

Note that the uniqueness statement in (2) is technically stronger than what is known for the K\"ahler-Ricci flow, where uniqueness is stated in terms of the potential function $\varphi$ solving a parabolic complex Monge-Amp\`ere equation \cite{ST4}.

Lastly, we consider the case when $\textrm{Kod}(M)=-\infty$ and $M$ is not minimal. In this case we still necessarily have $T<\infty$, but the total volume of $M$ may now also go to zero. If $\mathrm{Vol}(M,\omega(t))\geq c>0$ as $t\to T$ then the discussion is exactly the same as above. If on the other hand $\mathrm{Vol}(M,\omega(t))\to 0$ as $t\to T$, then $M$ must be either birational to a ruled surface, or of class VII (but not an Inoue) \cite[Theorem 1.5]{TW}. In fact, we can say that $M$ cannot be a blown-up Inoue surface either: if that was the case, there would be a sequence of blowups $\pi:M\to N$ where $N$ is Inoue, and so we would have $K_M=\pi^*K_N+E$ where $E$ is an effective $\pi$-exceptional divisor. Recall as mentioned earlier that $K_N$ admits a smooth Hermitian metric with nonnegative curvature, and $(K_N\cdot K_N)=0$. Then we would have
$$0\leq \int_E(\omega_0+2\pi Tc_1(K_M))=\int_E \omega_0+2\pi T(E\cdot E),$$
since this equals $\lim_{t\to T}\mathrm{Vol}(E,\omega(t))$, and so
\[\begin{split}
\lim_{t\to T}\mathrm{Vol}(M,\omega(t))&=\int_M(\omega_0+2\pi Tc_1(K_M))^2\\
&=\int_M\omega_0^2+4\pi^2T^2(E\cdot E)+4\pi T\int_M\omega_0\wedge \pi^*c_1(K_N)+4\pi T\int_E\omega_0\\
&\geq \int_M\omega_0^2+4\pi T\int_M\omega_0\wedge \pi^*c_1(K_N)+2\pi T\int_E\omega_0>0.
\end{split}\]
Apart from this, nothing is known about the behavior of the flow with finite-time collapsing singularities (see also the discussion above for Hopf surfaces and $\mathbb{P}^2$ and ruled surfaces). For the K\"ahler-Ricci flow, it is expected (and known in dimensions at most $3$ \cite{ToZ2}) that such singularities exist precisely when $M$ admits a {\em Fano fibration}.

\section{Finite time singularities} \label{sectfin}

Let $\omega(t)$ be a solution of the Chern-Ricci flow \eqref{CRF} on a compact complex manifold $M^n$, starting at a Hermitian metric $\omega$, which develops a singularity at a finite time $T$. As mentioned earlier, this condition is equivalent to $K_M$ not being nef. As in the K\"ahler case \cite{CT} we define the singularity formation locus as
$$
\Sigma=M\backslash  \left\{  \begin{array}{l}   x\in M\ |\ \exists U\ni x \textrm{ open, }\exists C>0,  \\
 \textrm{ s.t. } |\textrm{Rm}(\omega(t))|_{g(t)}\leq C \textrm{ on } U\times [0,T) \end{array} \right\}.
$$
Then Gill-Smith \cite{GS} show that we have
$$\Sigma=M\backslash  \left\{ \begin{array}{l} x\in M\ |\ \exists U\ni x \textrm{ open, }\exists C>0, \\  \textrm{ s.t. } R(\omega(t))\leq C \textrm{ on } U\times [0,T)\end{array} \right\},
$$
where $R(\omega(t))$ is the Chern scalar curvature, and furthermore also that
$$\Sigma=M\backslash \left\{ \begin{array}{l} x\in M\ |\ \exists U\ni x \textrm{ open, }\exists  \omega_T  \textrm{ Hermitian metric} \\ \textrm{on $U$  s.t. }   \omega(t)\overset{C^\infty(U)}{\to}\omega_T \textrm{ as }t\to T^{-}\end{array} \right\}.$$
In particular they deduce that at such a singularity the supremum of the Chern scalar curvature has to blow up, extending work of Zhang \cite{Z2} in the K\"ahler case.

Let now $V\subset M$ be an irreducible closed analytic subvariety with $\dim V=k>0$. Then its volume along the flow is
$$\mathrm{Vol}(V,\omega(t))=\int_V\frac{\omega(t)^k}{k!}>0,$$
and we can define
$$\Sigma'=\bigcup_{\mathrm{Vol}(V,\omega(t))\to 0}V,$$
as the set-theoretic union of all such subvarieties for which the volume shrinks to zero as $t$ approaches $T$ (the whole manifold $V=M$ is also allowed).
Then we have
\begin{equation}\label{inclusion}
\Sigma'\subset\Sigma.
\end{equation}
Indeed, suppose $x\in\Sigma'$, so there is a subvariety $V\ni x$ whose volume shrinks to zero. If $x\not\in\Sigma$, then there is some $U\ni x$ open and $\omega_T$ a Hermitian metric on $U$ such that $\omega(t)\to \omega_T$ smoothly on $U$ as $t$ approaches $T$. But then
$$\mathrm{Vol}(V,\omega(t))\geq\int_{V\cap U}\frac{\omega(t)^k}{k!}\to\int_{V\cap U}\frac{\omega_T^k}{k!}>0,$$
contradicting the fact that $\mathrm{Vol}(V,\omega(t))\to 0$. As observed in \cite{GS}, it follows from \cite{TW} that when $\dim M=2$ and $\omega_0$ is Gauduchon, then we have
\begin{equation}\label{sigma}
\Sigma'=\Sigma,
\end{equation}
and furthermore $\Sigma$ is equal to $M$ when $\mathrm{Vol}(M,\omega(t))\to 0$, and otherwise $\Sigma$ is a finite union of disjoint $(-1)$-curves. In particular, in this case $\Sigma$ is a closed analytic subvariety of $M$.

\begin{conjq}\label{cc}
Let $\omega(t)$ be a solution of the Chern-Ricci flow \eqref{CRF} on a compact complex manifold $M^n$, starting at a Hermitian metric $\omega$, which develops a singularity at finite time $T$. Then we have that \eqref{sigma} holds, and furthermore $\Sigma$ is a closed analytic subvariety of $M$.
\end{conjq}

Thanks to \eqref{inclusion}, the answer is affirmative in the case when the total volume $\mathrm{Vol}(M,\omega(t))$ goes to zero, so we can assume that the flow is volume noncollapsing. In the case when $\omega_0$ is K\"ahler, Question \ref{cc} was posed by Feldman-Ilmanen-Knopf \cite{FIK} and proved by Collins-Tosatti \cite{CT}. As mentioned above, Gill-Smith \cite{GS} proved it when $M$ is a surface and $\omega_0$ is Gauduchon. The first case to consider would then be when $M$ is a surface and $\omega_0$ is arbitrary.

\section{Singularity type at infinity} \label{sectinfin}

Let $M$ be a compact complex surface with $K_M$ nef, which as mentioned earlier is equivalent to the fact that the Chern-Ricci flow \eqref{CRF} starting at any Gauduchon metric $\omega_0$ has a solution for all $t\geq 0$.

As discussed earlier, Kodaira's classification shows that $M$ is one of the following:
\begin{itemize}
\item[(1)] A minimal surface of general type
\item[(2)] A minimal properly elliptic surface
\item[(3)] A K\"ahler Calabi-Yau surface
\item[(4)] A Kodaira surface
\item[(5)] An Inoue surface.
\end{itemize}

Following Hamilton \cite{Ha2}, we say that a long-time solution $\omega(t)$ of the Chern-Ricci flow \eqref{CRF} on $M$ develops a {\em type III} singularity at infinity if we have
$$\sup_{M\times[0,\infty)} t|{\rm Rm}(\omega(t))|_{g(t)}<+\infty,$$
or equivalently $\sup_M\left|{\rm Rm}\left(\frac{\omega(t)}{t}\right)\right|_{g(t)/t}$ remains uniformly bounded for all $t$ sufficiently large,
and a {\em type IIb} singularity if
$$\sup_{M\times [0,\infty)} t|{\rm Rm}(\omega(t))|_{g(t)}=+\infty,$$
or equivalently $\sup_M\left|{\rm Rm}\left(\frac{\omega(t)}{t}\right)\right|_{g(t)/t}$ does not have a uniform bound as $t$ goes to infinity.

In analogy with the results proved in \cite{ToZ} for the K\"ahler-Ricci flow, we have the following classification, which is complete modulo Conjectures \ref{prop_nk} and \ref{inoue} (3):
\begin{theorem}\label{sings}
Let $M$ be a compact complex surface with $K_M$ nef, and $\omega(t)$ a solution of the Chern-Ricci flow \eqref{CRF} on $M$ starting at a Gauduchon metric $\omega_0$. Then the flow develops a type III singularity at infinity if and only if:
\begin{itemize}
\item In case (1), if and only if $K_M$ is ample
\item In case (2), (assuming Conjecture \ref{prop_nk}) if and only if the only singular fibers are multiples of a smooth fiber (i.e. of Kodaira type $mI_0, m\geq 2$)
\item In case (3), if and only if $M$ is finitely covered by a torus and $\omega_0$ is K\"ahler
\item In case (4), never
\item In case (5), (assuming Conjecture \ref{inoue} (3)) always.
\end{itemize}
\end{theorem}

The key is the following Lemma, which is proved in \cite[Proposition 1.4]{ToZ} in the K\"ahler case, but whose proof works verbatim:
\begin{lemma}\label{tz}
Let $M$ be a compact complex manifold with $K_M$ nef, and suppose there is a nonconstant holomorphic map $f:\mathbb{P}^1\to M$ with $K_M\cdot f(\mathbb{P}^1)=0.$ Then every solution of the Chern-Ricci flow \eqref{CRF} on $M$ starting at any Hermitian metric $\omega_0$ develops a type IIb singularity at infinity.
\end{lemma}

As usual, we will refer to $f(\mathbb{P}^1)$ as a (possibly singular) rational curve in $M$.

\begin{proof}[Proof of Theorem \ref{sings}]
Assume case (1). If $K_M$ is ample, then \cite[Theorem 1.7]{TW} shows that $\omega(t)/t$ converges smoothly to a K\"ahler-Einstein metric on $M$, hence its curvature remains bounded, hence type III. If $K_M$ is not ample, then $M$ contains some $(-2)$-curve $C$, which satisfies $K_M\cdot C=0$, and so the flow is type IIb thanks to Lemma \ref{tz}.

Assume case (2). If there is some singular/multiple fiber which is not of type $mI_0, m\geq 2$, then by Kodaira's classification of these fibers \cite[\S V.7]{bhpv} this fiber contains a rational curve $C$. Since $C$ is contained in a fiber of $f$, and $K_M^\ell$ is the pullback of some divisor on $\Sigma$ for some $\ell\geq 1$ (thanks to Kodaira's canonical bundle formula \cite[Theorem V.12.1]{bhpv}), it follows that $K_M\cdot C=0$, and so the flow is type IIb thanks to Lemma \ref{tz}.

If on the other hand all singular fibers of $f$ are of type $mI_0,m\geq 2$, then $M$ is a ``quasi-bundle'', which is finitely covered by an elliptic bundle (see e.g. \cite[Proposition 3.3]{Ser}), and so the flow is type III thanks to Conjecture \ref{prop_nk}.

Assume case (3). If $\omega_0$ is K\"ahler, then this is proved in \cite[Theorem 1.5]{ToZ}. If $\omega_0$ is not K\"ahler, then by Gill's result \cite{G} we have that $\omega(t)$ converges smoothly to a Chern-Ricci-flat metric $\omega_\infty$ of the form $\omega_\infty=\omega_0+\ddbar\vp_\infty$ for some $\vp_\infty\in C^\infty(M)$. As mentioned earlier, the convergence is exponentially fast, and so it follows easily that the flow is type III if and only if $\omega_\infty$ is Chern-flat (i.e. its Chern Riemann curvature tensor vanishes identically). But a Hermitian metric on a compact complex surface is Chern-flat if and only if it is K\"ahler and flat by \cite{Gau}, which happens if and only if $M$ is finitely covered by a torus. But $\omega_\infty$ is K\"ahler if and only if so is $\omega_0$, and the result follows.

In case (4) the exact same discussion as case (3) applies, and of course in this case $\omega_\infty$ is never Chern-flat, since Kodaira surfaces are not finitely covered by a torus.

Lastly, in case (5), all solutions are type III thanks to Conjecture \ref{inoue} (3).
\end{proof}

To end this section, we pose an optimistic question which, if solved, would imply Conjectures \ref{prop_nk} and \ref{inoue} (2) and (3) (at least for Gauduchon initial metrics), and hence would make Theorem \ref{sings} unconditional:

\begin{conjq}\label{higherord}
Let $\omega(t)$ be a solution of the Chern-Ricci flow \eqref{CRF}
 on $B\times [0,1]$, where $B=B_1(0)\subset\mathbb{C}^n$. Suppose that
$$A^{-1}\omega_{\mathbb{C}^n}\leq \omega(t)\leq A \omega_{\mathbb{C}^n},$$
holds on $B\times[0,1]$, for some $A>0$. Then for every $k\geq 1$ there are constants $C=C(k,A,n)$ such that
$$\|\omega(t)\|_{C^k(B_{1/2},g_{\mathbb{C}^n})}\leq C,$$
for $t\in [\frac{1}{2},1]$.
\end{conjq}
For our purposes, it would be sufficient to prove this conjecture when $n=2$ and $\omega_0$ is Gauduchon. When $\omega(t)$ are K\"ahler, this conjecture is proved in \cite{ShW}. In the general case, partial results appear in \cite{ShW2} (see also \cite{Chu2,Nie2}), but the estimates proved there also depend on the initial metric $\omega_0$, which we do not allow here. The main issue is how to get control on covariant derivatives of the torsion of $\omega(t)$.

To see that  Question \ref{higherord} implies  Conjectures \ref{prop_nk} and \ref{inoue} (2), (3), it suffices to use the arguments in \cite{AT}, which we now outline. For elliptic bundles we take an arbitrary small ball $B\subset\Sigma$ over which the bundle is locally trivial, so that we can identify the preimage of $B$ with $F\times B$ where $F=\mathbb{C}/\Lambda$ is an elliptic curve, and let $p:\mathbb{C}\times B\to F\times B$ be the quotient map, with coordinates $(z_1,z_2)$ upstairs. Let also $\omega_F$ be semipositive $(1,1)$ form on $F\times B$ induced by $\sqrt{-1} dz_1\wedge d\ov{z}_1$. For an Inoue surface $M$, we let $p:\mathbb{C}\times\mathbb{H}\to M$ be the universal cover, and let $\omega_F$ be the semipositive $(1,1)$ form on $M$ which is denoted by $\beta$ in \cite{AT} (its explicit expression depends on the type of Inoue surface). Thus in both cases, upstairs we have an open subset of $\mathbb{C}^2$, where we fix a Euclidean metric. If we denote by $\omega_\infty$ the limit of $\omega(t)/t$ (for elliptic bundles this was denoted by $\omega_\Sigma$), then $p^*(\omega_\infty+\omega_F)$ is a Hermitian metric upstairs.

Instead of working with $\omega(t)/t$, it is more convenient to consider the Chern-Ricci flow $\omega(t)$ normalized by
\begin{equation}\label{NCRF}
\frac{\de \omega}{\de t}  =-\mathrm{Ric}(\omega)-\omega,
\end{equation}
which differs from our original $\omega(t)/t$ only by a space-time rescaling.
Then in either case we know from \cite{AT,FTWZ,TWY} that upstairs $p^*\omega(t)$ is uniformly equivalent to
$p^*(\omega_\infty+e^{-t}\omega_F).$
We then define stretching maps $\lambda_t$ upstairs by $\lambda_t(z_1,z_2)=(z_1e^{t/2},z_2)$, and metrics
$$\omega_t(s)=\lambda_t^*p^*\omega(s+t), \quad -1\leq s\leq 0,$$
which for all fixed $t\geq 0$ solve \eqref{NCRF} with time parameter $s$ and are uniformly equivalent to
$$\lambda_t^*p^*(\omega_\infty+e^{-s-t}\omega_F),$$
which one checks directly is uniformly comparable to a Euclidean metric, independent of $t\geq 0,-1\leq s\leq 0$. Applying Question \ref{higherord} (again after a harmless space-time rescaling to bring it back to the form \eqref{CRF}) thus gives
local uniform higher order estimates for $\omega_t(s)$, for $-\frac{1}{2}\leq s\leq 0$. Setting $s=0$ this gives local  uniform higher order estimates for $\lambda_t^*p^*\omega(t)$, which easily imply higher order estimates and the curvature bound for our original $\omega(t)/t$.

Of course, it is quite possible that Question \ref{higherord} turns out to be false, but we do believe that Conjectures \ref{prop_nk} and \ref{inoue} are true.

\section{Further questions and directions} \label{sectfurther}

\subsection{Higher dimensions and other flows}  \label{other} Although this survey focuses mainly on the case of complex surfaces, many results have been proved for the Chern-Ricci flow on higher dimensional compact complex manifolds, and with an arbitrary Hermitian initial metric $\omega_0$.  These include the results of Section \ref{sect_pde} on the maximal existence time, the convergence of the Chern-Ricci flow to a K\"ahler-Einstein metric on manifolds with $c_1(M)<0$ (see Section \ref{sectionkod2} and \cite{TW}) and Gill's result on manifolds with vanishing first Bott-Chern class (Section \ref{sectionKod0} and \cite{G}).

Examples of collapsing for Hopf surfaces, as described in Section \ref{sectionhopf} can also be extended to higher dimensional Hopf manifolds which are quotients of $\mathbb{C}^n \setminus \{ 0 \}$ \cite{TW3}.  Zheng \cite{Zheng} extended the work of Fang-Tosatti-Weinkove-Zheng \cite{FTWZ} on the Chern-Ricci flow on Inoue surfaces to their higher dimensional analogue, known as Oeljeklaus-Toma manifolds \cite{OT}.

On the other hand, it is not clear how to extend some other results to dimensions $n$ strictly larger than two.
  For example, one may replace exceptional curves of Section \ref{sectionnmcs} by the exceptional hypersurfaces one obtains by blowing up a point of an $n$-dimensional complex manifold (cf. \cite{SW1, SW2}) or consider more general birational transformations in the context of the minimal model program \cite{ST4, Ti}.  The behavior of the Chern-Ricci flow is not known in these cases.

  A major obstacle is that the Gauduchon condition $\partial \ov{\partial} \omega^{n-1}=0$ is weak and surely not preserved by the Chern-Ricci flow in general.  While the ``pluriclosed'' condition, $\partial \ov{\partial}\omega=0$,  \emph{is} preserved by the Chern-Ricci flow, such metrics do not exist on every compact complex manifold, and by itself may not be enough to extend the above-mentioned results.   A rather natural condition, considered in \cite{FT, GL, Shen, Shen2} for example, is to impose both $\partial \ov{\partial}\omega=0$ and $\partial \ov{\partial}\omega^2=0$ (which implies $\partial\ov{\partial}\omega^k=0$ for all $k$).  While the existence of such a metric restricts the possibilities for the underlying complex manifold, one can check that this condition is preserved by the Chern-Ricci flow (the authors thank Xi Sisi Shen for pointing this out).

Another approach is to consider flows of $(n-1, n-1)$ forms (or more generally, $(k,k)$ forms).  One such flow, related to the Gauduchon conjecture (proved in \cite{STW}) that one can prescribe the volume form of a Gauduchon metric \cite{Ga3}, was studied by Zheng \cite{Zheng3}, building on suggestions in \cite{G4, STW}.  Another flow of $(n-1,n-1)$ forms, known as the \emph{Anomaly flow}, is derived from the Hull-Strominger system in String Theory and has been extensively studied by Phong-Picard-Zhang, Fei-Phong and others (see \cite{FP, PPZsurvey, PPZ2} and the references therein).

Chu-Tosatti-Weinkove \cite{CTW} proposed an extension of the Chern-Ricci flow to manifolds with a non-integrable complex structure (see also \cite{KT,V} for different flows, the first one with varying almost complex structure).  Namely, let $\omega_0$ be an almost Hermitian metric and consider the flow:
\begin{equation} \label{ace}
\left\{
                \begin{aligned}
                  &\frac{\de \omega}{\de t}  =-\mathrm{Ric}^{(1,1)}(\omega),\\
                  &\omega(0)=\omega_0,
                \end{aligned}
              \right.
\end{equation}
 where $\textrm{Ric}^{(1,1)}(\omega)$ is the $(1,1)$ part of the Chern-Ricci form of $\omega$.  Chu \cite{Chu} proved long time existence and smooth convergence of (\ref{ace})  under the assumption of the existence of a metric $\omega$ with $\textrm{Ric}^{(1,1)}(\omega)=0$, which is the analogue of Gill's result for the Chern-Ricci flow \cite{G}.  Zheng \cite{Zheng2} then proved the analogue of the maximal existence time result of the authors \cite{TW} described in Section \ref{sect_pde} above.  These results used estimates from \cite{CTW} on the  Monge-Amp\`ere equation in the almost complex setting.

In a different direction, following La Nave-Tian \cite{LT} in the K\"ahler case, Sherman-Weinkove  \cite{ShW3} consider $\omega(s)$ solving an \emph{elliptic} version of the Chern-Ricci flow,
\begin{equation}  \label{continuity}
\omega(s) = \omega_0 - s \textrm{Ric}(\omega(s)), \quad s \ge 0,
\end{equation}
known as the \emph{continuity equation}.  It is shown in \cite{ShW3} that with a given initial Hermitian metric $\omega_0$ on a compact complex manifold, (\ref{continuity}) has a unique solution for $s \in [0,T)$ for the same $T$ as (\ref{T}).  This equation has the feature that the Chern-Ricci form is automatically bounded from below.  The equation (\ref{continuity}) has also been extended to the settings of almost Hermitian metrics \cite{LZ} and Gauduchon metrics \cite{Zhengc}.

\subsection{Curvature conditions}

Bando \cite{Ba} and Mok \cite{Mok} showed that the non-negativity of the bisectional curvature is preserved along the K\"ahler-Ricci flow.  More precisely, writing $R_{i\ov{j}k \ov{\ell}}$ for the components of the curvature tensor then (Griffiths) non-negativity asserts that at each point on the manifold,
$$R_{i\ov{j}k \ov{\ell}} X^i X^{\ov{j}} Y^k Y^{\ov{\ell}} \ge 0,$$
for all $(1,0)$ vectors $X, Y$.  It is natural then to ask whether similar curvature conditions may be preserved along the Chern-Ricci flow or other flows on complex manifolds.

The only known result along these lines for the Chern-Ricci flow is a negative one.   By considering Hopf manifolds, Yang \cite{Yang} gave an example to show that non-negativity of the bisectional curvature (with respect to the Chern connection) is \emph{not} preserved
along the Chern-Ricci flow.

On the other hand, Ustinovskiy \cite{U} considered the following example of the Hermitian curvature flow of Streets-Tian \cite{STi2},
\begin{equation} \label{u}
\left\{
                \begin{aligned}
                  &\frac{\de \omega}{\de t}  =-\mathrm{Ric}^{(2)}(\omega)+Q,\\
                  &\omega(0)=\omega_0,
                \end{aligned}
              \right.
\end{equation}
where $\mathrm{Ric}^{(2)}(\omega)$ is the \emph{second Chern-Ricci form} and $Q= \frac{1}{2}T * T$ is a certain quadratic term involving the torsion.  Ustinovskiy showed that (\ref{u})
 \emph{does} preserve the non-negativity of the bisectional curvature and used this to prove a Hermitian version of the Frankel conjecture \cite{U}.  Perhaps surprisingly, Ustinovskiy's flow coincides with a special case of the Anomaly flow \cite{FP}.

 Later, Lee \cite{Lee} showed  that, taking instead $Q=0$ in (\ref{u}), the  non-positivity of the Chern-Ricci curvature is preserved if the initial metric has \emph{non-positive} bisectional curvature.  He used this to show a compact Hermitian manifold with non-positive bisectional curvature and negative Chern-Ricci curvature at one point must have ample canonical bundle (and in particular, admit a K\"ahler metric).

While research on this topic is still at an early stage, the results so far suggest that flows involving the \emph{second} Chern-Ricci curvature seem to have desirable properties with respect to preservation of curvature conditions. On the other hand, unlike the Chern-Ricci flow such flows are not equivalent to scalar PDEs, and many results of the type that we have presented are not available (for example, a characterization of the maximal existence time).

Lastly, we mention here the work of Angella-Sferruzza \cite{AS}, who study whether the property of geometric formality of a metric (i.e. the wedge product of harmonic forms is still harmonic) is preserved by the Chern-Ricci flow.

\subsection{Noncompact manifolds}  The K\"ahler-Ricci flow on noncompact complete manifolds was first studied by Shi \cite{Shi1}.  It was  used  by Chen-Tang-Zhu \cite{CTZ}, Chau-Tam \cite{ChT}, Ni \cite{Ni} and others to make progress towards Yau's uniformization conjecture (which remains open as stated):  \emph{a complete noncompact K\"ahler manifold with positive holomorphic bisectional curvature is biholomorphic to $\mathbb{C}^n$}.    Earlier work on this conjecture was carried out by Mok-Siu-Yau \cite{MSY} and Mok \cite{Mok0}.

Liu \cite{Liu}, using non-flow methods, proved the strongest result known so far:  that a complete noncompact K\"ahler manifold with positive bisectional curvature and maximal volume growth is biholomorphic to $\mathbb{C}^n$.  In fact his result holds under the weaker hypothesis of nonnegative, rather than strictly positive, bisectional curvature.

Surprisingly,  Lee-Tam \cite{LeeTam} then used the Chern-Ricci flow to give a different (and almost contemporaneous) proof of Liu's theorem.   Lee-Tam construct a solution of the K\"ahler-Ricci flow starting from the given complete metric $g_0$ with possibly unbounded curvature at infinity.  The crux of their idea is that to deal with the possible bad behavior of the metric $g_0$ at infinity they make a conformal change to it outside some large compact set $S$.  The conformal change destroys the K\"ahlerity of the metric $g_0$ outside $S$, so instead they apply the \emph{Chern-Ricci flow} to this new metric which is only Hermitian outside of  $S$.  A crucial feature of the Chern-Ricci flow is that if a metric is K\"ahler in $S$ then it remains so along the flow.  By taking a family of exhausting sets $S$ they construct a global solution of the K\"ahler-Ricci flow, and then apply the result of Chau-Tam \cite{ChT} to conclude that the manifold is biholomorphic to $\mathbb{C}^n$.

In fact, Lee-Tam prove an analogue of the maximal existence time result of Section \ref{sect_pde} in the non-compact setting. These results were later improved and used in \cite{LeeTam2, Lott} to study Gromov-Hausdorff limits.  For some further work in the non-compact setting, see \cite{HLT}.

\end{document}